\documentclass{amsart}

\usepackage{amsmath, amsthm, amssymb, amscd}
\usepackage{ascmac}
\usepackage[all]{xy}
\usepackage{enumerate}
\usepackage{graphics}
\usepackage{color}

\newtheorem{theorem}{Theorem}[section]
\newtheorem{lemma}[theorem]{Lemma}
\newtheorem{prop}[theorem]{Proposition}

\theoremstyle{definition}

\theoremstyle{remark}
\newtheorem{rem}[theorem]{Remark}

\numberwithin{equation}{section}

\def\Im{\mathop{\mathrm{Im}}\nolimits}
\def\Ker{\mathop{\mathrm{Ker}}\nolimits}
\def\Coker{\mathop{\mathrm{Coker}}\nolimits}
\def\Hom{\mathop{\mathrm{Hom}}\nolimits}

\def\sgn{\mathop{\mathrm{sgn}}\nolimits}

\def\Aut{\mathop{\mathrm{Aut}}\nolimits}
\newcommand{\C}{\mathbb{C}}

\newcommand{\Z}{\mathbb{Z}}
\newcommand{\Q}{\mathbb{Q}}
\newcommand{\M}{\mathcal{M}}
\newcommand{\I}{\mathcal{I}}

\newcommand{\A}{\mathcal{A}}

\def\GL{\mathop{\mathrm{GL}}\nolimits}
\def\Sp{\mathop{\mathrm{Sp}}\nolimits}
\newcommand{\isoto}[1][]%
{{\mathop{\buildrel{\sim}\over\longrightarrow}\limits_{#1}}}
\def\L{\mathop{\mathcal{L}}\nolimits}

\def\C{\mathop{\mathcal{C}}\nolimits}

\def\gr{\mathop{\mathrm{gr}}\nolimits}
\def\Tr{\mathop{\mathrm{Tr}}\nolimits}
\def\Fab{\mathop{F_n^{\text{ab}}}}
\def\IA{\mathop{\mathrm{IA}}\nolimits}
\newcommand{\mf}[1]{{\mathfrak{#1}}}

\newcommand{\bb}[1]{{\mathbb{#1}}}

\begin{document}

\title[A comparison of classes in the Johnson cokernels]{A comparison of classes in \\ the Johnson cokernels of \\ the mapping class groups of surfaces}



\author{Naoya Enomoto}
\address{The University of Electro-Communications.}
\email{enomoto-naoya@uec.ac.jp}

\author{Yusuke Kuno}
\address{Tsuda University.}
\email{kunotti@tsuda.ac.jp}

\author{Takao Satoh}
\address{Tokyo University of Science.}
\email{takao@rs.tus.ac.jp}


\date{}



\begin{abstract}
In \cite{ES2}, the first and the third authors introduced new classes in the Johnson cokernels of the mapping class groups of surfaces by a representation theoretic approach based on some previous results for the Johnson cokernels of the automorphism groups of free groups.
On the other hand, in \cite{KK1}, Kawazumi and the second author introduced another type of classes by a topological consideration of self-intersections of curves on a surface.

In this paper, we show that the classes found in \cite{KK1} are contained in the classes found in \cite{ES2}
in a stable range.
Furthermore, we prove that
the anti-Morita obstructions $[1^{4m+1}]$ for $m \ge 1$ obtained in \cite{ES2} 
and a hook-type component $[3,1^5]$ detected in \cite{EE} appear in their gap.
\end{abstract}

\maketitle

\section{Introduction}

Let $\Sigma_{g,1}$ be a compact oriented surface of genus $g$ with one boundary component.
The mapping class group $\M_{g,1}$ is the group of isotopy classes of orientation preserving diffeomorphisms of $\Sigma_{g,1}$ which fix the boundary component pointwise.
The Torelli group $\I_{g,1}$, which consists of mapping classes acting trivially on the first homology $H=H_1(\Sigma_{g,1}, \bb{Z})$, is an important subgroup of $\M_{g,1}$.
There is a central filtration
$
\I_{g,1}=\M_{g,1}(1) \supset \M_{g,1}(2) \supset \M_{g,1}(3) \supset \cdots
$
defined by the action on the nilpotent quotients of the fundamental group of $\Sigma_{g,1}$.
The associated graded quotient of this filtration is described by the Johnson homomorphisms
\[
\tau_k^\M : \gr^k(\M_{g,1}) \hookrightarrow \mf{h}_{g,1}(k), \quad k\ge 1.
\]
Here, $\gr^k(\M_{g,1}) = \M_{g,1}(k)/\M_{g,1}(k+1)$ and $\mf{h}_{g,1}(k)$ is the kernel of the Lie bracket
$H \otimes_\bb{Z} \L_{2g}(k+1)\to \L_{2g}(k+2)$, where $\L_{2g} = \bigoplus_{m\ge 1} \L_{2g}(m)$ is the free Lie algebra generated by $H = \L_{2g}(1)$.
Note that the collection $\{ \tau_k^\M \}_k$ defines an injective homomorphism of graded Lie algebras:
\[
\tau^\M: \gr(\M_{g,1}) = \bigoplus_{k \ge 1} \gr^k(\M_{g,1}) \hookrightarrow
\mf{h}_{g,1} = \bigoplus_{k\ge 1} \mf{h}_{g,1}(k).
\]
The space $\mf{h}_{g,1}$ is called the Lie algebra of symplectic derivations \cite{Mo1, Ko1}.

The Johnson homomorphisms were introduced by Johnson \cite{Jo1, Jo2}, and Morita \cite{Mo1} gave a refinement of the target.
For recent developments in the theory of Johnson homomorphisms, we refer to expository articles \cite{HM, Hain19, KK2, MorSurvey, Sak, S3}.


A particularly important fact is that the map $\tau_k^\M$ is equivariant with respect to the action of the group $\M_{g,1}/\I_{g,1} \cong  \Sp(2g,\bb{Z})$.
This fact enables us to make use of representation theory to analyze $\tau_k^\M$, in particular when we work over a field of characteristic zero.
In what follows, putting $\bb{Q}$ as a subscript or a superscript means that one takes tensor product with the rationals.

As shown by Johnson \cite{Jo1} the first Johnson homomorphism $\tau_1^\M$ is surjective.
It was first observed by Morita \cite{Mo1} that the map $\tau_k^\M$ is not surjective for higher $k$.
That is, for any odd $k \ge 3$, he constructed the surjective homomorphism
$$\Tr_k : \mf{h}_{g,1}^\bb{Q}(k) \rightarrow S^kH_\bb{Q},$$
where $S^k$ means the $k$th symmetric tensor product, and proved that $\Tr_k \circ \tau_k^\M \equiv 0$.
In other words, the map $\Tr_k$ is an obstruction for the surjectivity of the $k$th Johnson homomorphism $\tau_k^\M$.
We call the quotient of $\mf{h}_{g,1}^{\bb{Q}}(k)$ by the image of $\tau_{k,\bb{Q}}^\M$ the $k$th Johnson cokernel of the mapping class group $\M_{g,1}$.
The $\Sp$-module structure of the Johnson cokernels becomes an interesting object of study.
The Morita trace $\Tr_k$ detects the unique $\Sp$-irreducible component $S^k H_{\bb{Q}}$ in the $k$th Johnson cokernel.

In \cite{ES2}, the first and the third authors introduced the $\Sp$-homomorphism
\[ c_k:\mf{h}_{g,1}^\bb{Q}(k) \rightarrow \C_{2g}^\bb{Q}(k). \]
(See \S \ref{subsec:ESobs} for its definition.)
Here, $\C_{2g}^\bb{Q}(k)$ is the quotient module of $H^{\otimes{k}}_\bb{Q}$ with respect to the action of the cyclic group of order $k$
as cyclic permutations of the components of $H^{\otimes{k}}_\bb{Q}$.
By using the third author's result in \cite{Sa}
that the space $\C_{2g}^\bb{Q}(k)$ coincides with the $k$th Johnson cokernel of the automorphism group of the free group,
they proved that
\[
\Im(\tau^\M_{k,\bb{Q}})\subset \Ker(c_k) \subset \mf{h}_{g,1}^\bb{Q}(k)
\]
in a stable range.
The map $c_k$ is a refinement of $\Tr_k$ in the sense that $\Ker(c_k)\subset \Ker(\Tr_k)$.
Moreover, in \cite{ES2} it was shown that for $k \equiv 1 \ (\text{mod} \ 4)$ and $k \ge 5$, an $\Sp$-irreducible component $[1^k]$ is detected in $\mf{h}_{g,1}^\bb{Q}(k)/\Ker(c_k)$, hence in the $k$th Johnson cokernel.
We call this component the anti-Morita obstruction.

There are several studies on the trace maps $c_k$ and their application to the Johnson cokernels.
In \cite{EE}, the first author and Hikoe Enomoto detected several series of hook-type components in $\mf{h}_{g,1}^\bb{Q}(k)/\Ker(c_k)$.
Recently, by using the hairy graph complex, Conant \cite{C} detected new $\Sp$-components in the Johnson cokernels which cannot be detected by the trace maps $c_k$.

At the present stage, the structure of the Johnson cokernels has not been completely determined.
By using the trace map $c_k$, Morita, Sakasai and Suzuki \cite{MSS} determined it up to degree $6$.

In \cite{KK1}, Kawazumi and the second author introduced the map
\[
\delta_k^{\text{alg}}:\mf{h}^\bb{Q}_{g,1}(k)\to \bigoplus_{\substack{p,q\ge 1, \\ p+q=k}}\C_{2g}^\bb{Q}(p)\otimes \C_{2g}^\bb{Q}(q)
\]
(See \S \ref{subsec:KKobs} for its definition.)
The map $\delta^{\text{alg}}_k$ arises from the Turaev cobracket, a topological operation which measures self-intersections of curves on a surface.
They showed that 
\[
\Im(\tau^\M_{k,\bb{Q}})\subset \Ker(\delta_k^{\text{alg}})\subset \mf{h}_{g,1}^\bb{Q}(k),
\]
and that $\Ker(\delta_k^{\text{alg}}) \subset \Ker(\Tr_k)$. 

The main purpose of this paper is to compare the two obstructions coming from $c_k$ and from $\delta_k^{\text{alg}}$.
Our first result is as follows.
\begin{theorem}\label{t1}
For each $k \ge 1$ and $2g \geq k+2$, we have $\Ker(c_k) \subset \Ker(\delta_k^{\text{alg}})$. 
\end{theorem}

Our proof is based on a relation between several contraction maps defined on $H^* \otimes_\Z \L_{2g}(k+1)$; see Theorem {\rmfamily \ref{tC}}.
We remark that recently, Alekseev, Kawazumi, Kuno and Naef \cite{AKKN} showed that the above theorem holds for any $g$ in a completely different way.

Our second result gives explicit differences between the two obstructions.

\begin{theorem}\label{t2}
Assume that $g\ge k+1$.
\begin{enumerate}[(i)]
\item For any $k \equiv 1 \ (\text{mod} \ 4)$ such that $k \ge 5$, the $\Sp$-irreducible component $[1^k]$ lies in $\Ker(\delta_k^{\text{alg}})/\Ker(c_k)$.
Thus $\Ker(c_k) \subsetneq \Ker(\delta_k^{\text{alg}})$.
\item For $k=8$, an $\Sp$-irreducible component $[3,1^5]$ appears in $\Ker(\delta_8^{\text{alg}})/\Ker(c_8)$.
\end{enumerate}
\end{theorem}

Topologically, each of the components in $\Ker(\delta_k^{\text{alg}})/\Ker(c_k)$ is a component of the $k$th Johnson cokernel, and cannot be detected by the usual Turaev cobracket, but by the framed version of it; see \cite{AKKN} and \cite{Ka}.
By some computer calculations, the first author and Hikoe Enomoto have checked that $[4,1^5]$ also appears in 
$\Ker(\delta_9^{\text{alg}})/\Ker(c_9)$. They conjecture that $[3,1^{k-3}] \ (5 \le k \equiv 0 \pmod{4})$ 
and $[4,1^{k-4}] \ (9 \le k \equiv 1 \pmod{4})$ appear in  $\Ker(\delta_k^{\text{alg}})/\Ker(c_k)$. 
These results and observations suggest that the difference of
$\Ker(\delta_k^{\text{alg}})$ and $\Ker(c_k)$ are not so small.

\section{Andreadakis-Johnson Theory for $\Aut{F_n}$}

In this section, we review the Andreadakis-Johnson filtration and the Johnson homomorphisms of the automorphism groups of free groups.
For details, see \cite{S2} for example.

\subsection{Johnson homomorphisms of $\Aut{F_n}$}
\label{subsec:JAut}

Let $F_n$ be a free group of rank $n \geq 2$ with basis $x_1, \ldots ,x_n$ and let $\Aut{F_n}$ be the automorphism group of $F_n$.
The group $\Aut{F_n}$ acts naturally on the abelianization $H:=\Fab:=F_n/[F_n,F_n]$ of $F_n$.
The kernel of this action is called the IA-automorphism group and denoted by $\IA_n$. 
The basis $x_1,\ldots,x_n$ induces a basis of $H$ and we can identify $\Aut{H}$ with the general linear group $\GL(n,\Z)$.
Thus we have the group extension
\[
1 \to \IA_n \to \Aut{F_n} \to \GL(n,\Z) \to 1.
\]

Let $F_n=\Gamma_n(1)\supset \Gamma_n(2) \supset \cdots$ be the lower central series of $F_n$. Namely it is defined by
$\Gamma_n(1):=F_n$ and $\Gamma_n(k):=[\Gamma_n(k-1),F_n]$ for $k \ge 2$.
It is classically known that the associated graded quotient 
$$\L_n:=\bigoplus_{k \ge 1}\L_n(k), \quad \text{where $\L_n(k):=\Gamma_n(k)/\Gamma_n(k+1)$},$$
has the graded Lie algebra structure induced from
the commutator bracket on $F_n$ and is isomorphic to the free Lie algebra generated by $H=\L_n(1)$.
Moreover, we have the canonical embedding
$$\L_n(k) \hookrightarrow H^{\otimes k}.$$

The group $\Aut{F_n}$ acts naturally on $F_n/\Gamma_n(k+1)$.
The kernel of this action is denoted by $\A_n(k)$.
Then the subgroups $\A_n(k)$ form the descending filtration $\IA_n=\A_n(1) \supset \A_n(2) \supset \cdots$ which we call
the Andreadakis-Johnson filtration.
Andreadakis proved the following theorem.
\begin{theorem}[Andreadakis \cite{And}] \label{tA} \ 
\begin{enumerate}[$(i)$]
\item For any $k,\ell \ge 1$, $\sigma \in \A_n(k)$ and $x \in \Gamma_n(\ell)$, we have $\sigma(x)x^{-1} \in \Gamma_n(k+\ell)$.
\item For any $k,\ell \ge 1$, we have $[\A_n(k),\A_n(\ell)]\subset \A_n(k+\ell)$, namely the Andreadakis-Johnson filtration $\{\A_n(k)\}$ is a descending central filtration of $\IA_n$.
\end{enumerate}
\end{theorem}

By Theorem \ref{tA} (i), for any $k \ge 1$ we can define the homomorphism 
$$\tilde{\tau}_k:\A_n(k) \to \Hom_\Z(H,\L_n(k+1))$$ 
by
\[
\sigma \mapsto \big( x \mod \Gamma_n(2) \mapsto \sigma(x)x^{-1} \mod \Gamma_n(k+2) \big).
\]
The kernel of $\tilde{\tau}_k$ coincides with $\A_n(k+1)$ and we obtain the injective homomorphism 
\[
\tau_k:\gr^k(\A_n) \hookrightarrow \Hom_\Z(H,\L_n(k+1))
= H^* \otimes_\Z \L_n(k+1),
\]
where $\gr^k(\A_n):=\A_n(k)/\A_n(k+1)$.
We call $\tau_k$ the $k$th Johnson homomorphism of $\Aut{F_n}$.

Next, we define a variant of the Johnson homomorphism $\tau_k$. Let
$
\IA_n=\A'_n(1) \supset \A'_n(2) \supset \cdots
$
be the lower central series of $\IA_n$, and set $\gr^k(\A'_n):=\A'_n(k)/\A'_n(k+1)$.
By Theorem \ref{tA} (ii), we have $\A'_n(k) \subset \A_n(k)$ for any $k$.
Thus we obtain the (not necessarily injective) homomorphism
\[
\tau'_k := \tau_k \circ i_k :\gr^k(\A'_n) \to H^* \otimes_\Z \L_n(k+1),
\]
where the map $i_k:\gr^k(\A'_n) \to \gr^k(\A_n)$ is induced from the inclusion $\A'_n(k) \hookrightarrow \A_n(k)$.

The group $\Aut{F_n}$ acts naturally on each graded quotient $\L_n(k)$.
Moreover, it acts on the normal subgroup $\A_n(k)$ by conjugation, and hence on the graded quotients $\gr^k(\A_n)$ and $\gr^k(\A_n')$.
The action of the subgroup $\IA_n$ on these quotients is trivial, and 
we obtain the well-defined action of the group $\GL(n,\Z)=\Aut{F_n}/\IA_n$ on $\L_n(k)$, $\gr^k(\A_n)$ and $\gr^k(\A_n')$.
The homomorphisms $\tau_k$ and $\tau'_k$ are $\GL(n,\Z)$-equivariant.

In \cite{Sa}, the third author completely determined the structure of the cokernels of $\tau'_k$ in a stable range.
Let $\C_n(k)$ be the quotient module of $H^{\otimes{k}}$ by the action of the cyclic group of order $k$.
Namely,
\[
\C_n(k):=H^{\otimes{k}}/\langle a_1 \otimes a_2 \otimes \cdots \otimes a_k-a_2 \otimes \cdots \otimes a_k\otimes a_1 | a_i \in H \rangle.
\]
One has $\C_n(0)=\Z$ and $\C_n(1)=H$.
Let 
$$\pi_k: H^{\otimes{k}} \to \C_n(k)$$
be the natural projection, and let
$\Phi_{12}: H^* \otimes_\Z H^{\otimes{k+1}} \to H^{\otimes{k}}$ be the contraction map defined by 
\[
\Phi_{12}( f \otimes a_1 \otimes a_2 \otimes \cdots \otimes a_{k+1})
= f(a_1)a_2 \otimes \cdots \otimes a_{k+1},
\]
where $f\in H^*$ and $a_i \in H$.
For simplicity, its restriction to $H^* \otimes_\Z \L_n(k+1)$ is denoted by the same letter: thus we obtain the map
$$\Phi_{12}: H^* \otimes _{\Z} \L_n(k+1) \to H^{\otimes k}.$$

\begin{theorem}[Satoh, \cite{Sa}]\label{thm-Satoh} \ Suppose $k \ge 2$ and $n \ge k+2$. 
\begin{enumerate}[$(i)$]
\item The homomorphism $\pi_k \circ \Phi_{12}: H^* \otimes_{\Z} \L_n(k+1) \rightarrow \C_n(k)$ is surjective.
\item We have $\Im\tau'_{k}=\Ker(\pi_k \circ \Phi_{12})$, namely $\Coker(\tau'_{k}) \cong \C_n(k)$.
\end{enumerate}
\end{theorem}
\noindent
Formulas of the $\GL$-irreducible decompositions of $\C_n^\Q(k)$ and $\Im(\tau'_{k,\Q})$ are given in \cite{ES1}.

\begin{rem}
Recently Darn\'{e} \cite{D} showed that the natural map $i_k:\gr^k(\A'_n) \to \gr^k(\A_n)$ is surjective
for $n \ge k+2$. This means that the stable $k$th cokernel $\mathrm{Coker}(\tau_k)$ coincides with $\Coker(\tau'_{k})$.
Namely, in the stable range, the Johnson cokernels for $\Aut{F_n}$ are completely determined over $\Z$.
\end{rem}

\subsection{A generating set of $\Im\tau_{k}'$}

Let $e_1, \ldots ,e_n$ be the standard basis of $H=\Fab$ induced from the basis $x_1, \ldots ,x_n$ of $F_n$, and $e_1^*, \ldots, e_n^*$ the dual basis of $H^*$.
For any $a_1, a_2, \ldots, a_k \in H$, we set
\[ [a_1, a_2, \ldots, a_k] := [ \cdots [[a_1, a_2], a_3], \ldots, ], a_k] \in \mathcal{L}_n(k). \]
This is called a $k$-simple commutator.
We have a generating set of $\Im\tau_{k}'$ as a $\Z$-module in a stable range.

\begin{prop}\label{T-Gen}
Suppose $k \ge 2$ and $n \ge k+2$.
Then the image of $\tau_{k}'$ is generated as a $\Z$-module by the following four types of elements in $H^* \otimes_{\Z} \L_n(k+1)$:
\begin{enumerate}[$(i)$]
\item[$(K_1)$] $e_i^* \otimes [e_{i_1},e_{i_2}, \ldots ,e_{i_{k+1}}]$
for any $1 \le i, i_1, \ldots , i_{k+1} \le n$ such that $i_1, \ldots , i_{k+1} \neq i$.
\item[$(K_2)$] $e_i^* \otimes [e_{i_1}, e_{i_2}, \ldots, e_{i_k}, e_i]$
for any $1 \le i, i_1, \ldots , i_k \le n$ such that $i_1, \ldots , i_k \neq i$.
\item[$(K_3)$] $e_i^* \otimes [e_i,e_{i_1}, \ldots ,e_{i_k}]-e_j^* \otimes [e_j,e_{i_k}, e_{i_1}, \ldots, e_{i_{k-1}}]$
for any $1 \le i, j, i_1, \ldots , i_k \le n$ such that $i,j \neq i_1, \ldots ,i_k$. (possibly $i=j$.) 
\item[$(K_4)$] $\displaystyle e_i^* \otimes [e_{i_1},e_{i_2}, \ldots ,e_{i_{k+1}}] -\sum_{j=1}^{k+1}\delta_{i,i_j}e_m^*\otimes [e_{i_1}, \ldots ,e_{i_{j-1}},e_m,e_{i_{j+1}}, \ldots ,e_{i_k},e_{i_{k+1}}]$
for any $1 \le i, m, i_1, \ldots , i_{k+1} \le n$ such that $i=i_j$ for some $1 \le j \le k+1$ and $m \neq i_1, \ldots ,i_{k+1}$.
\end{enumerate}
\end{prop}
\begin{proof}
It is easily seen that these elements belong to $\Ker(\pi_k \circ \Phi_{12})$.
In \S 3.2 in \cite{Sa}, it was shown that these elements belong to $\Im\tau_{k}'$.
Furthermore, by the arguments in the process of the proof of $\Im\tau'_{k} \supset \Ker(\pi_k \circ \Phi_{12})$,
it turns out that the above elements generate $\Ker(\pi_k \circ \Phi_{12})$ as a $\Z$-module.
Since $\Ker(\pi_k \circ \Phi_{12}) = \Im\tau'_{k}$, we obtain the required result.
\end{proof}

We remark that each of $\gr^k(\A'_n)$ is finitely generated since $\IA_n$ is finitely generated.
We should also remark that due to a recent work by Church, Ershov and Putman \cite{CEP}, each of $\mathcal{A}_n'(k)$ and $\mathcal{A}_n(k)$
is finitely generated in a stable range. However it seems to be still open to describe an explicit finite generating system of them.

\subsection{Contractions and $\Im\tau'_{k}$}

\label{sec:Cont}

We generalize the contraction map $\Phi_{12}$ in \S \ref{subsec:JAut}.
For each $1\le \ell \le k+1$, we consider the contraction map
$\Phi_{1,\ell+1}: H^* \otimes_\Z H^{\otimes{k+1}} \to H^{\otimes{k}}$ defined by the formula
\[
\Phi_{1,\ell+1}( f \otimes a_1 \otimes \cdots \otimes a_{k+1} )
= f(a_\ell) a_1 \otimes \cdots \otimes a_{\ell-1} \otimes a_{\ell+1} \otimes \cdots \otimes a_{k+1},
\]
where $f\in H^*$ and $a_i \in H$.
We denote its restriction to $H^* \otimes_\Z \L_n(k+1)$
by the same letter: thus we obtain the map
$$\Phi_{1,\ell+1}: H^* \otimes_\Z \L_n(k+1) \to H^{\otimes k}.$$ 

For $y \in H$ and $e_{i_j}$ for $1 \leq j \leq k$, 
in order to describe the expansion of the simple commutator $[y, e_{i_1}, \ldots, e_{i_k}]$ in $H^{\otimes k}$,
we introduce the following notation.
For an ordered subset $S=(j_1, j_2, \ldots, j_l)$ of the ordered set $(i_1, i_2, \ldots, i_k)$, define
\[\begin{split}
   e_{\overrightarrow{S}} := e_{j_1} \otimes e_{j_2} \otimes \cdots \otimes e_{j_l}, \hspace{1em}
   e_{\overleftarrow{S}} := e_{j_l} \otimes e_{j_{l-1}} \otimes \cdots \otimes e_{j_1}.
  \end{split}\]
Let $S^c$ be the ordered complement of $S$. For example, if $S$ is the ordered subset $(2,4,5)$ of $(1, 2, \ldots, 6)$, we have
$S^c =(1,3,6)$ and
\[\begin{split}
   e_{\overrightarrow{S}} & =e_{2} \otimes e_{4} \otimes e_{5}, \hspace{1em} e_{\overleftarrow{S}}=e_{5} \otimes e_{4} \otimes e_{2}, \\
   e_{\overrightarrow{S^c}} & =e_{1} \otimes e_{3} \otimes e_{6}, \hspace{1em} e_{\overleftarrow{S^c}}=e_{6} \otimes e_{3} \otimes e_{1}.
  \end{split}\]
Then, we have
\[ [y, e_{i_1}, \ldots, e_{i_k}] = \sum_S (-1)^{|S|} e_{\overleftarrow{S}} \otimes y \otimes e_{\overrightarrow{S^c}} \]
where $S$ ranges over all ordered subset of $(i_1, i_2, \ldots, i_k)$, and $|S|$ denotes the number of elements in $S$.
We can easily obtain the following lemma.
\begin{lemma}\label{L-HKT}
As notation above,
for any $1 \leq \ell \leq k+1$, if $i \neq i_1, i_2, \ldots, i_k$ then we have
\[ \Phi_{1,\ell+1}(e_i^* \otimes [e_i, e_{i_1}, \ldots, e_{i_k}]) = \sum_{\substack{S \subset (i_1, i_2, \ldots, i_k)\\[1pt] |S|=\ell-1}} (-1)^{l-1} e_{\overleftarrow{S}} \otimes e_{\overrightarrow{S^c}}. \]
\end{lemma}

For any $1\leq \ell \leq k+1$, 
define the homomorphism 
$$\varpi_\ell:H^{\otimes{k}} \to \C_n(\ell-1) \otimes \C_n(k-\ell+1)$$
by 
\[
\varpi_\ell(
a_1 \otimes \cdots \otimes a_{k+1})= \pi_{\ell-1}(a_1 \otimes \cdots \otimes a_{\ell-1}) \otimes \pi_{k-\ell+1}(a_\ell \otimes \cdots \otimes a_k),
\]
and set
\[
\Theta_\ell:=\varpi_\ell \circ \Phi_{1,\ell+1}:H^* \otimes_\Z H^{\otimes k+1} \to \C_n(\ell-1) \otimes \C_n(k-\ell+1). \]
We denote the restriction of this map to $H^* \otimes_\Z \L_n(k+1)$ by the same letter: 
$$\Theta_\ell: H^* \otimes_\Z \L_n(k+1) \to \C_n(\ell-1) \otimes \C_n(k-\ell+1).$$

\begin{theorem}\label{tC}
Suppose $k\ge 2$ and $n \geq k+2$.
For any $1 \leq \ell \leq k+1$, we have 
\[
\Ker(\Theta_1) \subset \Ker(\Theta_\ell)
\]
in $H^* \otimes_{\Z} \L_n(k+1)$. 
\end{theorem}
\begin{proof}
By Proposition {\rmfamily \ref{T-Gen}} and $\Ker(\Theta_1) = \Ker(\pi_k \circ \Phi_{12}) = \Im\tau'_{k}$, it suffices to show that
all the generators of type $K_1$, $K_2$, $K_3$ and $K_4$ of $\Im\tau'_{k}$ belong to $\Ker(\Theta_\ell)$ for any $1 \le \ell \le k+1$.
Clearly, generators of type $K_1$ belong to $\Ker(\Theta_\ell)$.
Consider a generator of type $K_2$. We have
\[\begin{split}
 \Phi_{1,\ell+1} & (e_i^* \otimes ([e_{i_1}, e_{i_2}, \ldots, e_{i_k}] \otimes e_i - e_i \otimes [e_{i_1}, e_{i_2}, \ldots, e_{i_k}])) \\
 &= \begin{cases}
      0 \hspace{1em} & \mathrm{if} \hspace{1em} \ell \neq 1, k+1, \\
      \pm [e_{i_1}, e_{i_2}, \ldots, e_{i_k}] & \mathrm{if} \hspace{1em} \ell = 1, k+1.
     \end{cases}
\end{split}\]
This shows that generators of type $K_2$ belong to $\Ker(\Theta_\ell)$,
since $\L_n(k)$ is in the kernel of the projection $\pi_k: H^{\otimes k} \to \C_n(k)$. 

For a generator
\[ X=e_i^* \otimes [e_i,e_{i_1}, \ldots ,e_{i_k}]-e_j^* \otimes [e_j,e_{i_k}, e_{i_1}, \ldots, e_{i_{k-1}}] \]
of type $K_3$. By Lemma {\rmfamily \ref{L-HKT}}, we have
\[ \Phi_{1,\ell+1}(X)=(-1)^{\ell -1} 
    \begin{bmatrix} \displaystyle \sum_{\substack{S \subset (i_1, i_2, \ldots, i_k)\\[1pt] |S|=\ell-1}} e_{\overleftarrow{S}} \otimes e_{\overrightarrow{S^c}}
    \,\, - \sum_{\substack{T \subset (i_k, i_1, \ldots, i_{k-1})\\[1pt] |T|=\ell-1}} e_{\overleftarrow{T}} \otimes e_{\overrightarrow{T^c}}
   \end{bmatrix}. \]
Here the first sum is written as
\[ \sum_{\substack{i_k \in S \\[1pt] |S|=\ell-1}} e_{\overleftarrow{S}} \otimes e_{\overrightarrow{S^c}}
  + \sum_{\substack{i_k \not\in S \\[1pt] |S|=\ell-1}} e_{\overleftarrow{S}} \otimes e_{\overrightarrow{S^c}}, \]
and the second sum is written as
\[ \sum_{\substack{i_k \in T \\[1pt] |T|=\ell-1}} e_{\overleftarrow{T}} \otimes e_{\overrightarrow{T^c}}
  + \sum_{\substack{i_k \not\in T \\[1pt] |T|=\ell-1}} e_{\overleftarrow{T}} \otimes e_{\overrightarrow{T^c}}. \]
Then we have
\[\begin{split}
  \sum_{\substack{i_k \in S \\[1pt] |S|=\ell-1}} & e_{\overleftarrow{S}} \otimes e_{\overrightarrow{S^c}}
    - \sum_{\substack{i_k \in T \\[1pt] |T|=\ell-1}} e_{\overleftarrow{T}} \otimes e_{\overrightarrow{T^c}} \\
  & = \sum_{\substack{S_0 \subset (i_1, \ldots, i_{k-1}) \\[1pt] |S_0|=\ell-2}} e_{i_k} \otimes e_{\overleftarrow{S_0}} \otimes e_{\overrightarrow{S_0^c}}
    - \sum_{\substack{T_0 \subset (i_1, \ldots, i_{k-1}) \\[1pt] |T_0|=\ell-2}} e_{\overleftarrow{T_0}} \otimes e_{i_k} \otimes e_{\overrightarrow{T_0^c}} \\
  & \overset{\varpi_{\ell}}{\longmapsto} 0
 \end{split}\]
since $\pi_{\ell-1}(e_{i_k} \otimes e_{\overleftarrow{S_0}}) = \pi_{\ell-1}(e_{\overleftarrow{S_0}} \otimes e_{i_k})$.
Similarly, the remaining terms
\[
\sum_{\substack{i_k \not\in S \\[1pt] |S|=\ell-1}} e_{\overleftarrow{S}} \otimes e_{\overrightarrow{S^c}}
-
\sum_{\substack{i_k \not\in T \\[1pt] |T|=\ell-1}} e_{\overleftarrow{T}} \otimes e_{\overrightarrow{T^c}}
\]
are annhilated by $\varpi_{\ell}$.
This shows that $\Phi_{1, \ell+1}(X)$ is in the kernel of $\varpi_\ell$ 
and thus generatos of type $K_3$ belong to $\mathrm{Ker}(\Theta_\ell)$ for any $1 \le \ell \le k+1$.

Finally, consider a generator
\[
X=e_i^* \otimes [e_{i_1},e_{i_2}, \ldots ,e_{i_{k+1}}]
   -\sum_{j=1}^{k+1}\delta_{i,i_j}e_m^*\otimes [e_{i_1}, \ldots ,e_{i_{j-1}},e_m,e_{i_{j+1}}, \ldots ,e_{i_k},e_{i_{k+1}}]
\]
of type $K_4$.
Assume $i_{j_1}= \cdots = i_{j_t} =i$. 
For any $1 \leq \ell \leq k+1$, we can calculate $\Phi_{1, \ell+1}(e_i^* \otimes [e_{i_1},e_{i_2}, \ldots ,e_{i_{k+1}}])$
by taking all contractions between $e_i^*$ and $e_{i_{j_s}}$ for $1 \leq s \leq t$. In particular, the contribution
of the contraction between $e_i^*$ and $e_{i_{j_s}}$ for a fixed $s$ is equal to that of
\[ \Phi_{1, \ell+1}(e_m^*\otimes [e_{i_1}, \ldots ,e_{i_{j_s -1}},e_m,e_{i_{j_s +1}}, \ldots ,e_{i_k},e_{i_{k+1}}]). \]
Therefore we see that $\Phi_{1,\ell+1}(X)=0$.
This completes the proof of Theorem {\rmfamily \ref{tC}}.
\end{proof}

\section{Structures of the Johnson Cokernels of $\M_{g,1}$}

In this section, we turn our attention to the mapping class group $\M_{g,1}$ and prove Theorem \ref{t1}
in Introduction.

\subsection{Johnson homomorphisms for $\M_{g,1}$}

\label{sec:JMG}

We review the Johnson homomorphisms and their cokernels of $\M_{g,1}$, following \cite{ES2}.
Given a base point on the boundary, 
the fundamental group $\pi_1(\Sigma_{g,1})$ of the surface $\Sigma_{g,1}$ is a free group $F_{2g}$ of rank $2g$.
Take a basis $x_1, x_2, \ldots, x_{2g}$ of $\pi_1(\Sigma_{g,1})$ such that the product $\prod_{i=1}^g [x_i,x_{i+g}]$ is parallel to the boundary component.
The homology classes $e_1, \ldots ,e_{2g}$ of $x_1, \ldots ,x_{2g}$ form a symplectic basis of the first homology group $H=H_1(\Sigma_{g,1},\Z)$.
The natural action of the mapping class group $\M_{g,1}$ on $\pi_1(\Sigma_{g,1})$ induces the Dehn-Nielsen embedding
\[
\varphi:\M_{g,1} \to \Aut(\pi_1(\Sigma_{g,1})) \cong \Aut{F_{2g}}.
\] 
Recall from \S \ref{subsec:JAut} the surjective homomorphism
$\pi:\Aut{F_{2g}} \rightarrow \GL(2g,\Z)$.
The image of $\pi_\M:=\pi \circ \varphi: \M_{g,1} \to \GL(2g,\Z)$ coincides with the integral symplectic group 
\[
\Sp(2g,\Z):=\{A \in \GL(2g,\Z) ; {}^tA J A=J\},
\]
where
$J=\left(
\begin{array}{cc}
0 & J_g \\
-J_g & 0
\end{array}
\right)$ and the $(g \times g)$-matrix
$J_g$ is equal to $\left(\begin{array}{ccc} O &  & 1 \\
                       & \rotatebox{75}{$\ddots$} &  \\
                     1 &         & O
     \end{array}\right)$.
The kernel of $\pi_\M$ is nothing but the Torelli group $\I_{g,1}$.
We obtain the following commutative diagram.
\[
\xymatrix{
1\ar[r] & \IA_{2g}\ar[r] & \Aut{F_{2g}}\ar[r]^{\pi} & \GL(2g,\Z)\ar[r] & 1 \\
1\ar[r] & \mathcal{I}_{g,1}\ar[r]\ar@{^(->}[u]^{\varphi|_{\mathcal{I}_{g,1}}} & \M_{g,1}\ar[r]_{\pi_\M}\ar@{^(->}[u]^{\varphi} & \Sp(2g,\Z)\ar[r]\ar@{^(->}[u] & 1
}
\]
For any $k\ge 1$ we set $\M_{g,1}(k):=\M_{g,1} \cap \A_{2g}(k)$, where $\A_{2g}(k)$ is the $k$th term of the Andreadakis-Johnson filtration of $\IA_{2g}$.
Let $\mathcal{I}_{g,1}=\M'_{g,1}(1) \supset \M'_{g,1}(2) \supset \cdots$ be the lower central series of $\mathcal{I}_{g,1}$.
Then we obtain the two homomorphisms
\[
\tau_k^\M:\gr^k\M_{g,1}=\M_{g,1}(k)/\M_{g,1}(k+1) \hookrightarrow H^* \otimes_\Z \L_{2g}(k+1)
\]
and
\[
{\tau'_k}^\M:\gr^k\M'_{g,1}=\M'_{g,1}(k)/\M'_{g,1}(k+1) \rightarrow H^* \otimes_\Z \L_{2g}(k+1)
\]
induced from the Dehn-Nielsen embedding and the Johnson homomorphisms of $\Aut{F_n}$.
We call $\tau_k^\M$ and ${\tau'_k}^\M$ the $k$th Johnson homomorphisms for $\M_{g,1}$. 
By an argument similar to that of the Johnson homomorphisms of $\Aut{F_n}$,
we see that the group $\Sp(2g,\Z)$ acts  naturally on the source and the target of the maps $\tau_k^\M$ and ${\tau'_k}^\M$,
and that $\tau_k^\M$ and ${\tau'_k}^\M$ are $\Sp(2g,\Z)$-equivariant
homomorphisms.
We remark that the homomorphism ${\tau'_k}^\M$ is not necessarily injective.
However, the following seminal work of Hain \cite{Hai} shows that the rational images of $\tau_k^\M$ and ${\tau'_k}^\M$ are equal.
\begin{theorem}[Hain \cite{Hai}]\label{thm-Hain}
We have $\Im\tau_{k,\Q}^\M=\Im{\tau'_{k,\Q}}^{\hspace{-3mm}\M}$ in $H_\Q^* \otimes_\Q \L_{2g}^\Q(k+1)$.
\end{theorem}
The space $H^*$ is canonically isomorphic to $H$ by the Poicar\'{e} duality and we can identify $H^* \otimes \L_{2g}(k+1)$ with $H \otimes \L_{2g}(k+1)$.
In \cite{Mo1}, Morita proved that $\Im\tau_{k}^\M \subset \mf{h}_{g,1}(k)$,
where $\mf{h}_{g,1}(k)$ is the kernel of the left bracketing homomorphism
\[
H \otimes \L_{2g}(k+1) \rightarrow \L_{2g}(k+2), \quad X\otimes u \mapsto [X,u].
\]

\subsection{Enomoto-Satoh's obstructions}
\label{subsec:ESobs}

In \cite{ES2}, Enomoto and Satoh introduced new classes in the Johnson cokernels.
These classes are defined by the $\Sp$-homomorphism
\[
c_k:\mf{h}^\Q_{g,1}(k) \hookrightarrow H_\Q \otimes_\Q \L_{2g}^\Q(k+1) \cong H_\Q^* \otimes_\Q \L_{2g}^\Q(k+1) \overset{\Theta_1}{\twoheadrightarrow} \C_{2g}^\Q(k),
\]
where $\Theta_1$ has been introduced in \S \ref{sec:Cont}. 
The following commutative diagram holds:
\[
\xymatrix{
& & \Im(\tau'_{k,\Q})\ar@{^(->}[rr] & & H_\Q^* \otimes_\Q \L_{2g}^\Q(k+1)\ar@{->>}[r]^{\hspace{2.5em} \Theta_1} & \C_{2g}^\Q(k) \\
\Im(\tau_{k,\Q}^\M)\ar@{=}[rr]^{\!\rm{Thm. \, \ref{thm-Hain}}} & & \Im({\tau'}_{k,\Q}^{\hspace{0.5mm}\M})\ar@{^(->}[r]\ar@{^(->}[u] & 
\mf{h}_{g,1}^{\Q}(k)\ar@{^(->}[r]\ar@<0.5ex>@{.>}[rru]^(.35){c_k} & H_\Q \otimes_\Q \L_{2g}^\Q(k+1)\ar@{=}[u]\ar@{->>}[r] & \L_{2g}^\Q(k+2)
}
\]
By using Theorem \ref{thm-Satoh} and Theorem \ref{thm-Hain}, they \cite{ES2} proved that 
\[
\Im(\tau_{k,\Q}^\M)\subset \Ker(c_k) \subset \mf{h}_{g,1}^\Q(k).
\]

\subsection{Kawazumi-Kuno's obstructions}
\label{subsec:KKobs}

In \cite{KK1}, Kawazumi and Kuno introduced another type of classes in the Johnson cokernels
by using some topological consideration on self-intersections of loops on the surface $\Sigma_{g,1}$.
In more detail, they considered an operation called the Turaev cobracket, and showed that its graded version $\delta^{\text{alg}}$ gives rise to an obstruction for the Johnson image.
(For more details, see \cite{KK1} and \cite{KK2}.)

The map $\delta^{\text{alg}}$ is homogeneous of degree $(-2)$ and the degree $k$ part
\[
\delta_k^{\text{alg}}: H_\Q^{\otimes k+2} \to \bigoplus_{\substack{p,q\ge 1, \\ p+q=k}}\C^\Q_{2g}(p) \otimes \C_{2g}^\Q(q)
\]
sends $a_1 \otimes \cdots \otimes a_{k+2}$ to
\[
\sum_{\substack{1 \le i<j \le k+2, \\ 1<j-i<k+1}}a_i^*(a_j)
\left\{
\begin{array}{l}
\pi(a_{i+1} \otimes \cdots \otimes a_{j-1})\otimes \pi(a_{j+1} \otimes \cdots \otimes a_{k+2} \otimes a_1 \otimes \cdots \otimes a_{i-1}) \\
-\pi(a_{j+1} \otimes \cdots \otimes a_{k+2} \otimes a_1 \otimes \cdots \otimes a_{i-1})\otimes \pi(a_{i+1} \otimes \cdots \otimes a_{j-1})
\end{array}
\right\}.
\]
Here, $a_i^* \in H^*_\Q$ is the element corresponding to $a_i\in H_\Q$ through the Poincar\'e duality $H^*_\Q=H_\Q$,
and $\pi$ denotes the projection $\pi_l: H_{\Q}^{\otimes l} \to \C^\Q_{2g}(l)$ when it is applied to $H_{\Q}^{\otimes l}$.
By restriction (and using the same letter), we obtain the map
\[
\delta_k^{\text{alg}}: \mf{h}_{g,1}(k) \to
 \bigoplus_{\substack{p,q\ge 1, \\ p+q=k}}\C^\Q_{2g}(p) \otimes \C_{2g}^\Q(q).
 \]
In \cite{KK1}, it was shown that
\[
\Im(\tau_k^\M) \subset \Ker(\delta^{\text{alg}}_k) \subset \mf{h}_{g,1}(k).
\]

\subsection{Proof of Theorem \ref{t1}}
\label{subsec:pft1}

Here we give a proof of Theorem \ref{t1}.
Recall from \S \ref{sec:Cont} the homomorphism $\Theta_\ell: H_\Q^* \otimes H_\Q^{\otimes k+1} \to \C_{2g}^\Q(\ell-1) \otimes \C_{2g}^\Q(k-\ell+1)$.
We can regard it as a map from $H_\Q^{\otimes k+2} = H_\Q \otimes H_\Q^{\otimes k+1}$ by the Poincar\'e duality.

\begin{proof}[Proof of Theorem \ref{t1}]
Let $\zeta$ be the cyclic permutation of the components of $H_\Q^{\otimes k+2}$ given by $\zeta(a_1\otimes a_2\otimes \cdots \otimes a_{k+2}):=a_2\otimes \cdots \otimes a_{k+2}\otimes a_1$ and set $\displaystyle \zeta_{k+2}:=\sum_{i=0}^{k+1} \zeta^i \in {\rm End}(H_\Q^{\otimes k+2})$.
Then, we see that
\[
\delta^{\text{alg}}_k =(\Theta_2+\cdots+\Theta_k) \zeta_{k+2}
\]
on $H^{\otimes k+2}_\Q$.
Since any element of $\mf{h}_{g,1}^\Q(k)$ is $\zeta$-invariant in $H_\Q^{\otimes k+2}$
(for instance, see \cite[Proposition 5.2]{ES2}), one has
$\delta^{\text{alg}}_k=(k+2)(\Theta_2+\cdots +\Theta_k)$ on $\mf{h}_{g,1}^\Q(k)$.


The homomorphism $\Theta_1$ is nothing but the trace map $c_k$, 
and hence $\Ker{c_k}=\Ker{\Theta_1}$.
By Theorem \ref{tC}, $\Ker{\Theta_1} \subset \Ker{\Theta_\ell}$ for any $\ell \ge 2$.
Therefore, $\Ker{c_k} \subset \Ker{\delta_k^{\mathrm{alg}}}$ on $\mf{h}_{g,1}(k)$.
\end{proof}

\begin{rem}
There is a refinement of $\delta^{\text{alg}}_k$ which uses the same formula but
we allow $j-i$ to be $1$ or $k+1$ so that $p$ and $q$ can be zero in the target.
This map comes from a framed version of the Turaev cobracket and does actually have
the same information as $c_k$.
For more detail, see \cite{AKKN} and \cite{Ka}.
\end{rem}

\section{Proof of Theorem \ref{t2}}

In this section, we prove Theorem \ref{t2}.
We consider polynomial representations of $\GL(2g,\bb{Q})$ and rational representations of $\Sp(2g,\bb{Q})$.
The isomorphism classes of $\GL$-irreducible polynomial representations are parametrized by partitions $\lambda$ such that their lengths $\ell(\lambda)$ are at most $2g$.
We denote by $(\lambda)$ the $\GL$-irreducible polynomial representation corresponding to a partition $\lambda$.
The isomorphism classes of $\Sp$-irreducible rational representations are parametrized by partitions $\lambda$ such that their lengths $\ell(\lambda)$ are at most $g$.
We denote by $[\lambda]$ the $\Sp$-irreducible rational representation corresponding to a partition $\lambda$.

\subsection{Anti-Morita obstruction $[1^k]$}
In this subsection, we prove Theorem \ref{t2}(i).

First, we recall the anti-Morita obstruction. In \cite{ES2}, we have the following result. 
\begin{theorem}[Enomoto and Satoh {\cite[Theorem 1]{ES2}}]
Suppose $g \ge k+1$ and $k \equiv 1 \pmod{4}$ and $k \ge 5$.
The multiplicities of $\Sp$-irreducible representations $[1^k]$ are exactly one in 
$\mf{h}_{g,1}^\bb{Q}(k)/\Ker(c_k)$.
\end{theorem}

We also recall the $\GL$-irreducible decomposition of $\C^\Q_{2g}(k)$ 
obtained by \cite{ES1}.

\begin{lemma}[{\cite[Corollary 4.2(2)]{ES1}}]\label{lem:c}
Suppose $2g \ge k$. 
The multiplicity $[\C_{2g}^\bb{Q}(k):(1^k)]$ of $(1^k)$ in $\C_{2g}^{\bb{Q}}(k)$ 
is equal to 1 if k is odd, and 0 if k is even.
\end{lemma}
\begin{proof}[Proof of Theorem \ref{t2}(i).] 
Note that $g \ge k+1$ implies $2g \ge k$.
Assume $k \equiv 1 (\mathrm{mod} \ 4)$ and $k \ge 5$.

To prove that the $\Sp$-homomorphism $\delta_k^{\text{alg}}: \mf{h}_{g,1}(k) \to \bigoplus_{\substack{p,q\ge 1, \\ p+q=k}}\C^\Q_{2g}(p) \otimes \C_{2g}^\Q(q)$
annhilates the $\Sp$-irreducible component $[1^k]$ in $\mf{h}_{g,1}(k)/\Ker(c_k)$, 
it is sufficient to show that $[1^k]$ does not appear in all $\C^\Q_{2g}(p) \otimes \C_{2g}^\Q(q)$ for $p,q \ge 1$ and $p+q=k$.
By the $\GL$-$\Sp$ branching rule, 
it is enough to show that there is no $\GL(2g, \bb{Q})$-irreducible representation $(1^k)$ in 
$\C_{2g}^\bb{Q}(p) \otimes \C_{2g}^{\bb{Q}}(q)$ for $p+q=k$ and $p,q\ge 1$.

For partitions $\mu$ and $\nu$ of $p$ and $q$ respectively, suppose $(\mu) \otimes (\nu)$ has the $\GL$-irreducible representation $(1^k)$.
If $\ell(\mu)<p$ or $\ell(\nu)<q$, we have $\ell(\mu)+\ell(\nu)<k$. Then by the Littlewood-Richardson rule, there is no $\GL$-irreducible representation $(1^k)$ in $(\mu) \otimes (\nu)$. Hence, we consider $\ell(\mu)=p$ and $\ell(\nu)=q$.
This case is nothing but $\mu=(1^p)$ and $\nu=(1^q)$. Since $p+q=k \equiv 1 \pmod{4}$, the signatures of $p$ and $q$ are different.
By Lemma \ref{lem:c}, there is no component $(1^p) \otimes (1^q)$ in $\C_{2g}^\bb{Q}(p) \otimes \C_{2g}^{\bb{Q}}(q)$. This is a contradiction.
\end{proof}
\begin{rem}\label{rem:c}
Especially, for $5\le k \equiv 1 \pmod{4}$ and $g\ge k+1$,
an $\Sp$-irreducible component $[1^k]$ appears in $\Ker(\Theta_2)/\Ker(c_k)$,
thus $\Ker(c_k) \neq \Ker(\Theta_2)$.
\end{rem}
\subsection{A hook-type component $[3,1^5]$}
\quad In this subsection, we prove Theorem \ref{t2} (ii). \\
\quad First, in \cite[Theorem 1.1]{EE},
several series of hook-type Sp-irreducible components $[r+1,1^{k-r-1}]$ 
are detected in $k$th Johnson cokernel $\mf{h}_{g,1}(k)/\Ker(c_k)$.
An $\Sp$-irreducible representation $[3,1^5]$ for $k=8$ and $r=2$ is one of such components.
\begin{prop}
For $g\ge 9$, an $Sp$-irreducible component $[3,1^5]$ appears in $\mf{h}_{g,1}(8)/\Ker(c_8)$.
\end{prop}
Note that the multiplicity of $[3,1^5]$ is larger than or equal to $1$ in each 
$\C_{2g}^\bb{Q}(p) \otimes \C_{2g}^{\bb{Q}}(8-p)$ for $1 \le p \le 7$. 
Therefore, to prove that $[3,1^5]$ lies in 
$\Ker(\delta_8^{\text{alg}})/\Ker(c_8)$,
we need to use a different way from the previous subsection.
We consider a maximal vector which gives a component $[3,1^5]$ in $\mf{h}_{g,1}(8)/\Ker(c_8)$ and 
prove that it lies in $\Ker(\Theta_\ell)$ for $2 \le \ell \le 8$.

As in \S \ref{sec:JMG} we fix a symplectic basis $\{e_1, \ldots ,e_g,e_{g+1}, \ldots ,e_{2g}\}$ of $H_\bb{Q}$.
Set $i':=2g-i+1$ for each integer $1 \le i \le 2g$.
We see that
\begin{eqnarray*}
\langle e_i,e_j\rangle=0=\langle e_{i'},e_{j'}\rangle,\quad 
\langle e_i,e_{j'}\rangle=\delta_{ij}=-\langle e_{j'},e_{i}\rangle, \quad (1 \le i,j \le g). 
\end{eqnarray*}
For each integer $1 \le i \le 2g$, we define 
$
e_i^*=\left\{
\begin{array}{ll}
e_{i'}, & (1 \le i \le g), \\
-e_{i'}, & (g+1 \le i \le 2g).
\end{array}
\right.
$
Then
$\langle e_i,e_j^* \rangle=\delta_{ij}$ for any $i,j$. 
Set $\omega=\displaystyle \sum_{i=1}^{2g}e_i \otimes e_i^* \in H_{\bb{Q}}^{\otimes{2}}$. 
We identify $H_\bb{Q}$ with $H_\bb{Q}^*$ by $v \mapsto \langle v,\bullet \rangle$. 
Note that $\langle e_{r'},e_r\rangle e_{r'}^*=e_r$ for $1 \le r \le 2g$. \\
\quad We define 
\[
v_{[3,1^5]}:=\omega \otimes (e_1 \wedge e_2 \wedge e_3 \wedge e_4 \wedge e_5 \wedge e_6) \otimes e_1 \otimes e_1
\in H_{\Q}^{\otimes 10}
\]
where $e_1 \wedge e_2 \wedge \cdots \wedge e_6$
is the anti-symmetrizer $\displaystyle \sum_{\sigma \in \mf{S}_6}\sgn(\sigma) (e_{\sigma(1)} \otimes e_{\sigma(2)} \otimes \cdots \otimes e_{\sigma(6)}) 
\in H_{\bb{Q}}^{\otimes{6}}$. \\
\quad Let $s_i$ be the permutation of $i$ and $i+1$. 
By the Brauer-Schur-Weyl duality, the set of elements 
$\{v_{[3,1^5]} \cdot \tau \cdot \theta \cdot \zeta_{10} \ (\tau \in \mf{S}_{10})\}$ generates 
the space of $\Sp$-maximal vectors corresponding to $\Sp$-irreducible components $[3,1^5]$ in $\mf{h}_{g,1}(8)$, 
where $\theta=(1-s_2)(1-s_3s_2) \cdots (1-s_9s_8 \cdots s_2)$ is the Dynkin-Specht-Wever idempotent and 
$\zeta_{10} \in {\rm End}(H_\Q^{\otimes 10})$
is defined in the proof of Theorem \ref{t1}. \\
\quad In \cite{EE}, a component $[3,1^5]$ is detected in $\mf{h}_{g,1}(8)/\Ker(c_8)$ by proving the following claim.
\begin{prop}[{\cite[Proposition 3.8]{EE}}]
$c_8(v_{[3,1^5]}\theta\zeta_{10}) \neq 0$.
\end{prop}
Recall from \S \ref{subsec:pft1} that, up to scalar, $\delta_{8}^{\text{alg}}$ is equal to $\Theta_2+ \cdots +\Theta_8$.
Therefore the following theorem implies Theorem \ref{t2} (ii).
\begin{theorem}
For $2 \le \ell \le 8$, we have $\Theta_\ell(v_{[3,1^5]}\theta\zeta_{10})=0$.
\end{theorem}
\begin{proof}
Note that it is sufficient to prove the claim for $\ell=2,3,4,5$. 
We use the following notations. The $(i,j)$-expansion operator $D_{ij}:H_\bb{Q}^{\otimes{k}} \to H_\bb{Q}^{\otimes k+2}$ is given by
\[
(v_1 \otimes \cdots \otimes v_k)D_{ij}=
\sum_{r=1}^{2g}v_1 \otimes \cdots \otimes v_{i-1} \otimes e_r \otimes v_i \otimes \cdots \otimes v_{j-2} \otimes e_r^* \otimes v_{j-1} \otimes \cdots \otimes v_k.
\]
The element $\Lambda_{a,b}\in H_{\Q}
^{\otimes 8}$ is given by 
\[
\sum_{\sigma \in \mf{S}_6}\mathrm{sgn}(\sigma)
e_{\sigma(1)} \otimes \cdots \otimes e_{\sigma(a-1)} \otimes e_1 \otimes e_{\sigma(a)} \otimes \cdots \otimes e_{\sigma(b-2)} \otimes e_1 \otimes e_{\sigma(b-1)} \otimes \cdots \otimes e_{\sigma(6)}.
\]
In \cite[Proposition 3.3]{EE}, we have
\begin{eqnarray*}
v_{[3,1^5]}\theta&=&
(e_1 \wedge \cdots \wedge e_6) \otimes e_1^{\otimes 2} \cdot (D_{12}-3D_{14}+3D_{16}-D_{18}) \\
&{}&+e_1 \otimes (e_1 \wedge \cdots \wedge e_6) \otimes e_1 \cdot (-2D_{13}+6D_{15}-6D_{17}+2D_{19}) \\
&{}&+e_1^{\otimes 2} \otimes (e_1 \wedge \cdots \wedge e_6) \cdot (D_{14}-3D_{16}+3D_{18}-D_{1,10}).
\end{eqnarray*}
Let us denote the three terms in the right hand side by
$\boldsymbol{v}_1,\boldsymbol{v}_2$ and $\boldsymbol{v}_3$. \\
\quad For the $13$-contraction operator $\Phi_{13}$, we obtain 
\begin{eqnarray*}
&{}&\Phi_{13}(\boldsymbol{v}_1)=2\Lambda_{1,8}+2\Lambda_{6,7}-2\Lambda_{2,3}+3\Lambda_{1,4}-3\Lambda_{1,6}, \\
&{}&\Phi_{13}(\boldsymbol{v}_2)=(-4g-2)\Lambda_{1,8}+(-4g-2)\Lambda_{1,2}-4\Lambda_{6,8}+4\Lambda_{2,4}, \\
&{}&\Phi_{13}(\boldsymbol{v}_3)=2\Lambda_{1,2}-2\Lambda_{3,4}+2\Lambda_{7,8}-3\Lambda_{1,4}+3\Lambda_{1,6}.
\end{eqnarray*}
Then we have 
\[
\Phi_{13}(v_{[3,1^5]}\theta\zeta_{10})=
-(4g)(\Lambda_{1,2}+\Lambda_{1,8})
+2(\Lambda_{6,7}-\Lambda_{2,3})+4(\Lambda_{2,4}-\Lambda_{6,8})+2(\Lambda_{7,8}-\Lambda_{3,4}).
\]
The first term is in the kernel of $\varpi_2:H_\bb{Q}^{\otimes{8}} \to C_{2g}(1) \otimes C_{2g}(7)$ because 
$\Lambda_{1,2}$ and $\Lambda_{1,8}$ are of the form 
$e_1 \otimes (\text{a maximal vector with weight} \ (2,1^5) \ \text{in} \ H_{\bb{Q}}^{\otimes{7}})$, and $(2,1^5)$ does not appear in $C_{2g}(7)$ 
(\cite[Corollary 4.2]{ES2}).
The remaining three terms are also in the kernel of $\varpi_2$ because they cancel each other in $C_{2g}(1) \otimes C_{2g}(7)$.
Hence we obtain $v_{[3,1^5]}\theta\zeta_{10} \in \Ker(\Theta_2)$. \\
\quad For the $14$-contraction operator $\Phi_{14}$, we have 
\begin{eqnarray*}
\Phi_{14}(v_{[3,1^5]}\theta\zeta_{10})&=&
(4g)\Lambda_{1,2}+(-6g+1)(\Lambda_{3,4}+\Lambda_{7,8}) \\
&{}& -2\Lambda_{3,5}-2\Lambda_{4,5}+6\Lambda_{4,6}-6\Lambda_{5,6}+6\Lambda_{5,7}-2\Lambda_{6,7}-2\Lambda_{6,8}.
\end{eqnarray*}
The first term is in the kernel of $\varpi_3$ because $[1^6]$ does not appear in $C_{2g}(6)$. All the other terms are in the kernel of $\varpi_3$ 
because each term is contained in $e_i \otimes e_j \otimes v-e_j \otimes e_i \otimes v \in H_\bb{Q}^{\otimes{8}}$.
Hence we obtain $v_{[3,1^5]}\theta\zeta_{10} \in \Ker(\Theta_3)$. \\
\quad For the $15$-contraction operator $\Phi_{15}$, we have 
\begin{eqnarray*}
\Phi_{15}(v_{[3,1^5]}\theta\zeta_{10})&=&12g(\Lambda_{1,8}+\Lambda_{3,4}) \\
&{}&+4(\Lambda_{4,5}-\Lambda_{7,8})+4(\Lambda_{5,6}-\Lambda_{6,7})+8(\Lambda_{6,8}-\Lambda_{4,6})
.
\end{eqnarray*}
In the first term, by dividing
$\sum_{\sigma \in \mf{S}_6}$ into $\sum_{\sigma(1) \ \text{or} \ \sigma(2)=1}
+\sum_{\sigma(3) \ \text{or} \ \sigma(6)=1}
+\sum_{\sigma(4) \ \text{or} \ \sigma(5)=1}$, they are in the kernel of $\varpi_4$.
The remaining three terms are also in the kernel of $\varpi_4$ because they cancel each other in $C_{2g}(3) \otimes C_{2g}(5)$.
Thus we obtain $v_{[3,1^5]}\theta\zeta_{10} \in \Ker(\Theta_4)$. \\
\quad For the $16$-contraction operator $\Phi_{16}$, we have 
\begin{eqnarray*}
\Phi_{16}(v_{[3,1^5]}\theta\zeta_{10})&=&-(6g+1)(\Lambda_{1,2}+\Lambda_{3,4}-\Lambda_{5,6}-\Lambda_{7,8}) \\
&{}& +2(\Lambda_{6,7}-\Lambda_{2,3}+\Lambda_{2,4}-\Lambda_{6,8}+\Lambda_{1,3}-\Lambda_{5,7})
.
\end{eqnarray*}
Since the projection $H_\bb{Q}^{\otimes{4}} \to \C_{2g}^\Q(4)$ annihilates the elements $e_i \wedge e_j \wedge e_k \wedge e_\ell$,
all the terms are in the kernel of $\varpi_5$.
Therefore we obtain $v_{[3,1^5]}\theta\zeta_{10} \in \Ker(\Theta_5)$. 
\end{proof}

\section*{Acknowledgments}\label{S-Ack}

The first author is supported by JSPS KAKENHI 26870368 and 18K03204.
He also would like to thank Hikoe Enomoto for his careful support on some computer calculations.
The second author is supported by JSPS KAKENHI 26800044 and 18K03308.
The third author is supported by JSPS KAKENHI 24740051 and 16K05155.

\bibliographystyle{amsplain}

\end{document}